\documentclass[12pt]{article}
\usepackage{amsmath}
\usepackage{amssymb}
\usepackage{mathrsfs}
\usepackage{amsthm}
\usepackage{authblk}
\usepackage{color}
\usepackage{graphicx}
\usepackage{caption}
\usepackage{subcaption}
\usepackage{makecell}
\usepackage{soul}
\usepackage{bbm} 
\usepackage{ulem}
\usepackage{cancel} 
\usepackage{algorithm}
\usepackage{algorithmicx}
\usepackage{amssymb}
\usepackage{threeparttable}
\usepackage{titlesec}
\usepackage{tikz}
\usepackage{bbding}
\usepackage{pifont}
\usepackage{titletoc}
\usepackage{hhline}
\usepackage{bbm}
\usepackage{adjustbox}

\theoremstyle{definition}
\newtheorem{definition}{Definition}
\newtheorem{theorem}{Theorem}
\newtheorem{lemma}{Lemma}

\newtheorem{corollary}{Corollary}

\newtheorem{fact}{Fact}
\usepackage{hyperref}
\usepackage[margin=1.2in]{geometry}

\usepackage{hyperref}

\newcommand{\x}{\pmb{x}}
\newcommand{\y}{\pmb{y}}

\newcommand{\vt}{\pmb{t}}
\newcommand{\vs}{\pmb{u}}
\newcommand{\veps}{\pmb{\varepsilon}}
\newcommand{\veta}{\pmb{\eta}}

\newcommand{\X}{\pmb{X}}

\newcommand{\E}{\mathbb{E}}

\newcommand{\R}{\mathbb{R}}

\newcommand{\mS}{\mathscr{U}}
\newcommand{\mA}{\mathscr{A}}
\newcommand{\mT}{\mathscr{T}}

\newcommand{\mQ}{\mathscr{Q}}
\newcommand{\gw}{w}
\newcommand{\we}{\widetilde{w}}
\newcommand{\Prob}{\mathbb{P}}

\def\lV{\left\lVert}
\def\rV{\right\lVert}
\def\lv{\left\lvert}
\def\rv{\right\lvert}
\def\lk{\left(}
\def\rk{\right)}
\def\lg{\langle}
\def\rg{\rangle}
\def\lz{\left[}
\def\rz{\right]}

\title{Focused Width in Adversarial Fake Detection: A Separation}
\author{Gao Huang
\footnote{School of Mathematical Science, Zhejiang University, Hangzhou 310027, P. R. China, E-mail address: hgmath@zju.edu.cn}}
\date{}

\begin{document}

\maketitle

\begin{abstract}
We study the adversarial fake detection model of Mendelson, Paouris and Vershynin~\cite{MPV}. 
In this model, a genuine sample is $\X\sim N\lk0,\pmb{I}_n\rk$, while a fake sample is produced as $\X+r\vt\lk{\X}\rk$, where the adversary first observes $\X$ and then chooses an admissible perturbation $\vt\lk{\X}\rk$ from a prescribed set $\mT\subset\R^n$. 
The central quantity is the detectability radius $r\lk\mT\rk$, which formalizes the transition scale at which fake samples become reliably distinguishable from genuine ones.
Mendelson, Paouris and Vershynin introduced the focused width $\we\lk\mT\rk$ as a geometric parameter for this radius and conjectured that, for every origin-symmetric set $\mT$, it characterizes $r\lk\mT\rk$ up to universal constants.

In this note, we disprove this conjecture for a broad class of discrete sets.
More precisely, we consider any origin-symmetric set $\mT_n$ lying between the hypercube and the odd integer grid:
\begin{equation*}
\{-1,1\}^n\subset \mT_n\subset \lk 2\mathbb{Z}+1\rk^n.
\end{equation*}
For every such $\mT_n$, we prove that 
$\frac{\we\lk\mT_n\rk}{r\lk\mT_n\rk }\gtrsim \sqrt{\log n}$. 
Thus, in the Gaussian model, the focused width can overestimate the detectability radius by a $\sqrt{\log n}$ factor and therefore does not characterize it in general.
We further show that this logarithmic scale is not intrinsic: in the corresponding non-Gaussian model with product Laplace data, the focused width benchmark can even exceed the detectability radius by at least a polynomial factor of order $n^{1/4}$.
\end{abstract}


\section{Introduction}

The rapid rise of generative AI has made fake data easy to produce and difficult to detect.
Even in the stylized setting of Gaussian data, deciding whether a sample is genuine or fake can already be subtle.
A key distinction is whether the adversary chooses the perturbation before or after observing the realized sample.
In the outsider setting, the perturbation, or more generally the alternative distribution, is fixed before the sample is drawn.
This viewpoint underlies much of classical Gaussian detection theory, including minimax signal detection~\cite{Baraud2002,chhor2024sparse}, sparse mixture and higher criticism testing~\cite{Donoho2004,Cai2007,Arias02011,arias2019detection,donoho2024impossibility}, sparse regression~\cite{Ingster2010,Carpentier2019}, and combinatorial or network detection~\cite{Addario2010,Arias2011}.
By contrast, in the insider adversarial model introduced by Mendelson, Paouris and Vershynin~\cite{MPV}, the adversary first observes the sample and only then chooses the perturbation.
A related sample-dependent phenomenon also appears in Smirnov's work~\cite{Smirnov2024}.
This adaptivity is the central feature absent from the outsider formulation: it changes the geometry of the testing problem and can make detection substantially harder.

In the model of~\cite{MPV}, a genuine data point is $\X\sim N\lk0,\pmb{I}_n\rk$ in $\R^n$.
Given a prescribed perturbation set $\mT\subset\R^n$, the adversary observes $\X$ and then selects an admissible perturbation, or trick, $\vt\lk\X\rk\in\mT$, which may depend on the realized sample.
The adversary then releases
\begin{equation*}
\X+r\vt\lk{\X}\rk,
\end{equation*}
where $r>0$ measures the corruption scale.
The goal of the adversary is to use this sample dependence to hide the corruption and make $\X+r\vt\lk\X\rk$ statistically indistinguishable from $\X$.
The tester observes only the released point and must decide whether it is a genuine Gaussian sample or was produced by this adversarial procedure.
Thus a successful test should accept most genuine samples while rejecting adversarially shifted samples uniformly over all admissible sample-dependent choices of $\vt\lk\X\rk$.

For small values of $r$, the perturbation $r\vt\lk\X\rk$ is close to the origin, so the corrupted sample is difficult to distinguish from a genuine Gaussian sample.
For large values of $r$, the corruption is expected to create a detectable deviation from the Gaussian distribution.
Thus, for a fixed perturbation set $\mT$, the natural quantitative question is to identify the critical scale at which reliable detection becomes possible, uniformly over all sample-dependent choices of $\vt\lk\X\rk\in\mT$.
This scale is formalized by the \textbf{\textit{detectability radius}}; see~\cite{MPV}.


To define this radius precisely, we represent a test by a measurable set  $\mA\subset\R^n$: points in $\mA$ are declared genuine, and points in $\mA^c$ are declared fake. 
We say that a scale $r>0$ is detectable for $\mT$ if there exists such a set $\mA$ satisfying
\begin{equation}\label{eq:detectable}
\gamma_n\lk\mA\rk\ge 0.9,
\qquad
\gamma_n\lk\mA-r\mT\rk\le 0.1,
\end{equation}
where $\gamma_n$ denotes the standard Gaussian measure on $\R^n$ and
\begin{equation*}
\mA-r\mT:=\left\{\x-r\vt:\x\in\mA,\ \vt\in\mT\right\}.
\end{equation*}
The first condition means that most genuine Gaussian samples are accepted. 
The second condition controls the success probability of any admissible insider adversary:
indeed, $\X\in\mA-r\mT$ if and only if there exists some $\vt\in\mT$ such that $\X+r\vt\in\mA$.
Thus, if $\gamma_n\lk\mA-r\mT\rk\le 0.1$, then no sample-dependent choice of trick can move the original Gaussian sample into the acceptance region with probability larger than $0.1$.

The detectability radius of $\mT$ then cis defined by
\begin{equation}\label{eq:radius}
r\lk\mT\rk
:=\sup\left\{r>0:\text{there is no measurable }\mA\subset\R^n \text{ satisfying }~\eqref{eq:detectable}\right\}
\end{equation}
Thus, $r\lk\mT\rk$ is the supremum of corruption scales for which no test of the above form can reliably distinguish genuine Gaussian samples from adversarially shifted samples.

A central result of~\cite{MPV} relates this transition scale to a rescaled Gaussian width.
For a general set $\mT\subset\R^n$, define
\begin{equation}\label{eq:scaled_width}
\bar{\gw}\lk\mT\rk=
\E\sup_{\vt\in\mT}
\left\langle \pmb{g},\frac{\vt}{\lV\vt\rV_2^2}\right\rangle,
\qquad \pmb{g}\sim N\lk0,\pmb{I}_n\rk .
\end{equation}
This quantity is a rescaled version of the usual Gaussian width~\cite{Vershynin2018}; the normalization by $\lV\vt\rV_2^2$ gives greater weight to tricks closer to the origin, which are harder to detect.
For highly symmetric sets $\mT$, Mendelson, Paouris and Vershynin~\cite{MPV} showed that this scaled Gaussian width characterizes the detectability radius up to universal constants.
More precisely, in their normalization,
 \begin{equation*}
r\lk\mT\rk\approx2\bar{\gw}\lk\mT\rk.
\end{equation*}
Here, highly symmetric means that whenever $\mT$ contains a point $\x$, it also contains every vector $\y$ with the same support and no smaller Euclidean norm, namely
\begin{equation*}
\mathrm{supp}\lk\y\rk=\mathrm{supp}\lk\x\rk,
\qquad
\lV\y\rV_2\ge \lV\x\rV_2.
\end{equation*}
Typical examples are the Euclidean exterior
$\left\{\x\in\R^n:\lV\x\rV_2\ge 1\right\}$
and its $s$-sparse analogue
$\left\{\x\in\R^n:\lv \mathrm{supp}\lk\x\rk\rv\le s,\ \lV\x\rV_2\ge 1\right\}$.

The high-symmetry assumption cannot be omitted.
This is illustrated in Section~5.1 of~\cite{MPV} by a simple example for which the scaled Gaussian width has the wrong order of magnitude.
For $n\ge2$, consider
\begin{equation}\label{eq:1/2}
\mT=\left\{\vt\in\R^n:\lV\vt\rV_2=1,\ \lv t_1\rv=1/2\right\}.
\end{equation}
This set is origin-symmetric  but not highly symmetric.
Moreover, $\bar{\gw}\lk\mT\rk\asymp \sqrt n$, whereas $r\lk\mT\rk=O\lk 1\rk$.
Indeed, every admissible trick shifts the first coordinate by $\pm r/2$, so a one-dimensional test based only on this coordinate detects the perturbation once $r$ is a sufficiently large constant.
Thus, outside the highly symmetric regime, the scaled Gaussian width can dramatically overestimate the true detectability radius.


This failure motivates the~\textbf{\textit{focused width}} introduced in~\cite{MPV}.
\begin{definition}\label{def:focused_width}
For $\mT\subset\R^n$, define
\begin{equation}\label{eq:focused_width}
\we\lk\mT\rk
:=\inf_{\mS}\gw\lk\mS\rk,
\end{equation}
where, for each set $\mS\subset\R^n$,
\begin{equation*}
\gw\lk\mS\rk
:=
\E\sup_{\vs\in\mS}\left\langle \pmb{g},\vs\right\rangle,
\qquad
\pmb{g}\sim{N}\lk0,\pmb{I}_n\rk,
\end{equation*}
and the infimum is taken over all origin-symmetric sets $\mS\subset\R^n$ satisfying the following hitting condition:
\begin{equation}\label{eq:hit}
\forall\, \vt\in\mT\quad \exists\, \vs\in\mS
\quad\text{with}\quad
\left\langle \vt,\vs\right\rangle\ge 1.
\end{equation}
\end{definition}

The set $\mS$ plays the role of a family of detecting directions: the hitting condition~\eqref{eq:hit} requires every trick $\vt\in\mT$ to have a nontrivial projection onto at least one direction in $\mS$.
Thus the focused width measures the Gaussian complexity of a family of detecting witnesses, rather than the Gaussian complexity of the entire rescaled perturbation set.
This refinement is designed to ignore irrelevant directions that may inflate the scaled Gaussian width.
In the example~\eqref{eq:1/2}, this improvement is substantial: $\we\lk\mT\rk=\mathcal{O}\lk1\rk$, whereas $\bar{\gw}\lk\mT\rk\asymp\sqrt{n}$.

This refinement led Mendelson, Paouris and Vershynin to pose the following geometric characterization problem:
does the focused width determine the detectability radius, up to universal constants, for every origin-symmetric perturbation set?
Equivalently, is it true that
\begin{equation}\label{eq:conjecture}
r\lk\mT\rk\asymp \we\lk\mT\rk
\end{equation}
uniformly over all origin-symmetric sets $\mT\subset\R^n$?
Since Theorem~5.3 of~\cite{MPV} shows that sufficiently large multiples of $\we\lk\mT\rk$ are detectable for every origin-symmetric set $\mT$, the conjecture would require the converse lower bound $r\lk\mT\rk\ge c\,\we\lk\mT\rk$ for some universal constant $c>0$ and every origin-symmetric set $\mT\subset\R^n$.

However, we disprove this \textit{focused width conjecture} for a broad class of discrete perturbation sets lying between the hypercube and the odd integer grid.
Let $\mQ_n=\left\{-1,1\right\}^n$ be the discrete hypercube, and let $\mT_n$ be any origin-symmetric set satisfying
\begin{equation}\label{eq:set}
\mQ_n\subset \mT_n\subset \lk2\mathbb{Z}+1\rk^n.
\end{equation}
Our main result is the following.

\begin{theorem}\label{thm:main}
There exists a universal constant $C>0$ such that, for every origin-symmetric set $\mT_n$ satisfying~\eqref{eq:set} and for all sufficiently large $n$,
\begin{equation*}
\frac{\we\lk\mT_n\rk}{r\lk\mT_n\rk}
\ge C\sqrt{\log n}.
\end{equation*}
\end{theorem}

Theorem~\ref{thm:main} rules out the conjectured universal lower bound in~\cite{MPV}: for these origin-symmetric discrete perturbation sets, the focused width overestimates the true detectability radius by at least a factor of order $\sqrt{\log n}$.


We prove Theorem~\ref{thm:main} in Section~\ref{sec:main}.
The proof has two ingredients.
The first is a dimension-free lower bound on the focused width for the class of perturbation sets considered here:
$\we\lk\mT_n\rk\gtrsim 1$.
The second ingredient is a periodic detection mechanism that is invisible from the linear geometric viewpoint underlying the focused width.
Inspired by Smirnov's reduction of hypercube detection to adversarial coin testing~\cite{Smirnov2024}, we consider the testing function
\begin{equation*}
F_r\lk\x\rk=\sum_{j=1}^n \cos\lk \frac{\pi x_j}{r}\rk .
\end{equation*}
Every shift $r\vt$ with $\vt\in\lk2\mathbb{Z}+1\rk^n$ moves each coordinate by an odd number of half periods and therefore flips the sign of $F_r$.
Thus the same periodic test works simultaneously for every trick vector in any set $\mT_n\subset\lk2\mathbb{Z}+1\rk^n$.
A Gaussian concentration argument then shows that this periodic test yields detectability for corruption scales at most of order $1/\sqrt{\log n}$.
Combining this periodic detection estimate with the dimension-free lower bound on $\we\lk\mT_n\rk$ proves Theorem~\ref{thm:main}.

The logarithmic separation in Theorem~\ref{thm:main} reflects a feature of the Gaussian setting rather than an intrinsic limitation of adversarial fake data detection.
In Section~\ref{sec:nongaussian}, we show that, for non-Gaussian data, the focused width prediction can fail by a substantially larger margin.
More precisely, we consider product distributions and prove that, for product Laplace data, the focused width benchmark can exceed the detectability radius by at least a polynomial factor of order $n^{1/4}$.


\section{Proof of Main Result}\label{sec:main}

\subsection{Focused Width}

We first compute the focused width of all sets lying between the hypercube and the odd integer grid.

\begin{lemma}\label{lm:focused}
For every set $\mT_n$ satisfying $\mQ_n\subset \mT_n\subset \lk2\mathbb{Z}+1\rk^n$, one has
\begin{equation}\label{eq:focused_cube}
    \we\lk\mT_n\rk=\sqrt{\frac{2}{\pi}}.
\end{equation}
\end{lemma}

\begin{proof}
We first prove the lower bound. 
Let $\mS\subset\R^n$ be an arbitrary origin-symmetric set satisfying the hitting condition~\eqref{eq:hit} with $\mT=\mT_n$.
Define its support function by
\begin{equation*}
    h_{\mS}\lk \x\rk=\sup_{\vs\in\mS}\lg \x,\vs\rg.
\end{equation*}
We may assume $\gw\lk\mS\rk<\infty$. 
The support function $h_{\mS}$ is convex and positively homogeneous; moreover, since $\mS=-\mS$, it is nonnegative.

Since $\mQ_n\subset\mT_n$, the hitting condition for $\mT_n$ implies
\begin{equation}\label{eq:1}
    h_{\mS}\lk\veps\rk\ge 1
\end{equation}
for every $\veps\in\mQ_n$.
Let $\veta=\lk\eta_1,\ldots,\eta_n\rk$ be a random Rademacher vector, and let
$\pmb{\alpha}=\lk\alpha_1,\ldots,\alpha_n\rk$, where $\alpha_1,\ldots,\alpha_n$ are independent copies of $\lv N\lk0,1\rk\rv$, independent of $\veta$. 
Then
\begin{equation*}
    \pmb{g}\stackrel{d}{=}\veta\odot\pmb{\alpha},
    \qquad
\pmb{g}\sim{N}\lk0,\pmb{I}_n\rk,
\end{equation*}
where $\odot$ denotes coordinatewise multiplication.

Conditioning on $\veta$ and applying Jensen's inequality to the convex function $h_{\mS}$, we obtain
\begin{equation}\label{eq:E}
\begin{aligned}
\E_{\pmb{\alpha}} h_{\mS}\lk\veta\odot\pmb{\alpha}\rk
&\ge h_{\mS}\lk \E_{\pmb{\alpha}}\lk\veta\odot\pmb{\alpha}\rk\rk  \\
&= h_{\mS}\lk \sqrt{\frac{2}{\pi}}\,\veta\rk= \sqrt{\frac{2}{\pi}}\, h_{\mS}\lk\veta\rk ,
\end{aligned}
\end{equation}
where we used $\E\lv N\lk0,1\rk\rv=\sqrt{2/\pi}$ and the homogeneity of $h_{\mS}$. 
Taking expectation with respect to $\veta$ and using~\eqref{eq:1} gives
\begin{equation*}
\begin{aligned}
\gw\lk\mS\rk
=\E h_{\mS}\lk\pmb{g}\rk
&=\E_{\veta}\E_{\pmb{\alpha}} h_{\mS}\lk\veta\odot\pmb{\alpha}\rk\\
&\ge \sqrt{\frac{2}{\pi}}\,\E_{\veta} h_{\mS}\lk\veta\rk
\ge \sqrt{\frac{2}{\pi}}.
\end{aligned}
\end{equation*}
Since $\mS$ was arbitrary, taking the infimum over all origin-symmetric sets satisfying the hitting condition gives
\begin{equation}\label{eq:focused_lower}
    \we\lk\mT_n\rk\ge \sqrt{\frac{2}{\pi}}.
\end{equation}

It remains to prove the matching upper bound. 
Take $\mS_0=n^{-1}\mQ_n$.
Since $\mQ_n=-\mQ_n$, the set $\mS_0$ is origin-symmetric.
We claim that $\mS_0$ satisfies the hitting condition for $\mT_n$.
Indeed, for every $\vt=\lk t_1,\ldots,t_n\rk\in\mT_n$, the inclusion $\mT_n\subset\lk2\mathbb{Z}+1\rk^n$ implies that each $t_j$ is a nonzero odd integer, and hence $\lv t_j\rv\ge 1$.
Choose
\begin{equation*}
\vs=\frac{1}{n}\operatorname{sgn}\lk\vt\rk\in n^{-1}\mQ_n,
\end{equation*}
where $\operatorname{sgn}\lk\vt\rk:=\lk\operatorname{sgn}\lk t_1\rk,\ldots,\operatorname{sgn}\lk t_n\rk\rk$.
Then
\begin{equation*}
\lg \vt,\vs\rg
=
\frac{1}{n}\sum_{j=1}^n \lv t_j\rv
\ge 1.
\end{equation*}
Thus $\mS_0$ satisfies the hitting condition for $\mT_n$.
Therefore
\begin{align*}
\we\lk\mT_n\rk
\le \gw\lk n^{-1}\mQ_n\rk&=\frac{1}{n}\E\sup_{\veps\in\{-1,1\}^n}\lg \pmb{g},\veps\rg 
=\frac{1}{n}\E\sum_{i=1}^n \lv g_i\rv =\sqrt{\frac{2}{\pi}}.
\end{align*}
Combining this upper bound with~\eqref{eq:focused_lower} proves~\eqref{eq:focused_cube}.
\end{proof}

\subsection{A Periodic Detection Test}

We now prove an upper bound for the detectability radius of any perturbation set lying between the hypercube and the odd integer grid.
The construction is inspired by Smirnov's periodic coin testing argument for hypercube attacks~\cite{Smirnov2024}.
In that argument, a periodic partition of the real line turns the Gaussian detection problem into a coin testing problem in which the adversary flips many labels.
Here we use a smooth Fourier analogue of this idea: the periodic observable $\cos\lk \pi x/r\rk$ changes sign under every shift by an odd integer multiple of $r$.
Fix $r>0$. 
Define the testing function by
\begin{equation}
F_r\lk \x\rk
=\sum_{j=1}^n \cos\lk \frac{\pi x_j}{r}\rk,
\end{equation}
and define the acceptance region
\begin{equation*}
\mA_r=\{\x\in\R^n:F_r\lk \x\rk>0\}.
\end{equation*}

\begin{lemma}\label{lm:detect}
Let $\mT_n\subset\R^n$ satisfy $\mQ_n\subset \mT_n\subset \lk2\mathbb{Z}+1\rk^n$.
Let $\X\sim N\lk 0,\pmb{I}_n\rk$.
Let $u>1$ and assume that $n>2\log u$.
If
\begin{equation}\label{eq:r}
r\ge
\frac{\pi}{\log^{1/2}\lk \frac{n}{2\log u}\rk},
\end{equation}
then the acceptance region $\mA_r$ defined above satisfies
\begin{equation}\label{eq:measure}
\Prob\left\{ \X\in \mA_r\right\}\ge 1-\frac{1}{u}
\quad \text{and}\quad
\Prob \left\{ \X\in\mA_r-r\mT_n\right\} \le \frac{1}{u}.
\end{equation}
\end{lemma}

The proof of Lemma~\ref{lm:detect} is based on two ingredients. 
The first ingredient is a deterministic sign-flipping property of $F_r$.

\begin{fact}\label{prop:flip}
For every $\x\in\R^n$ and every $\vt\in\lk2\mathbb{Z}+1\rk^n$,
\begin{equation}\label{eq:flip}
F_r\lk \x+r\vt\rk=-F_r\lk \x\rk.
\end{equation}
Consequently, for every nonempty set $\mT_n\subset\lk2\mathbb{Z}+1\rk^n$,
\begin{equation}\label{eq:minkowski}
\mA_r-r\mT_n=\{\x\in\R^n:F_r\lk \x\rk<0\}.
\end{equation}
\end{fact}

\begin{proof}
For each coordinate $j$, since $t_j\in2\mathbb{Z}+1$, we have
\begin{equation*}
\begin{aligned}
\cos\lk \frac{\pi\lk x_j+r t_j\rk}{r}\rk
=\cos\lk \frac{\pi x_j}{r}+\pi t_j\rk =-\cos\lk \frac{\pi x_j}{r}\rk .
\end{aligned}
\end{equation*}
Summing this identity over $j=1,\dots,n$ gives \eqref{eq:flip}.

It remains to identify the shifted set. 
Fix $\x\in\mathbb{R}^n$. 
By the definition of this shifted set, $\x\in\mA_r-r\mT_n$ if and only if there exists $\vt\in\mT_n$ such that $\x+r\vt\in\mA_r$.
This is equivalent to requiring
$F_r\lk\x+r\vt\rk>0$ for some $\vt\in\mT_n$. 
By~\eqref{eq:flip}, for every $\vt\in\mT_n$ we have $F_r\lk\x+r\vt\rk=-F_r\lk\x\rk$.
Since $\mT_n$ is nonempty, the above condition is equivalent to $F_r\lk\x\rk<0$.
This proves~\eqref{eq:minkowski}.

\end{proof}

The second ingredient is a concentration estimate for $F_r\lk \X\rk$ under the standard Gaussian measure.

\begin{fact}\label{prop:concentration}
Let $r>0$ and let $\X\sim N\lk 0,\pmb{I}_n\rk$. 
Then
\begin{equation}\label{eq:prob}
\Prob\lk F_r\lk \X\rk\le 0\rk
\le
\exp\lz-\frac{n}{2}\exp\lk-\frac{\pi^2}{r^2}\rk\rz.
\end{equation}
\end{fact}

\begin{proof}
Let
\begin{equation*}
Z_j=\cos\lk \frac{\pi X_j}{r}\rk,
\qquad j=1,\dots,n.
\end{equation*}
Then the random variables $\{ Z_j\}_{j=1}^n$ are independent and take values in $[-1,1]$. 
By the characteristic function of a standard Gaussian variable, we obtain
\begin{equation*}
\begin{split}
\E Z_j
&=\E\cos\lk \frac{\pi X_j}{r}\rk
=\E\Re\lz \exp\lk \mathrm{i}\frac{\pi X_j}{r}\rk \rz\\
&=\Re\lz \E\exp\lk \mathrm{i}\frac{\pi X_j}{r}\rk \rz
=\exp\lk-\frac{\pi^2}{2r^2}\rk
=:\nu.
\end{split}
\end{equation*}
Hence, $\E F_r\lk \X\rk=n\nu$.
Applying Hoeffding's inequality~\cite[Theorem 2.2.6]{Vershynin2018} to the centered variables $\{ Z_j-\nu\}_{j=1}^n$ gives
\begin{equation*}
\begin{aligned}
\Prob\lk F_r\lk \X\rk\le 0\rk
&=\Prob\lk\sum_{j=1}^n \lk Z_j-\nu\rk\le-n\nu\rk\\        
&\le\exp\lk-\frac{n\nu^2}{2}\rk                          
=\exp\lz-\frac{n}{2}\exp\lk -\frac{\pi^2}{r^2}\rk\rz.
\end{aligned}
\end{equation*}
\end{proof}

We are now ready to prove Lemma~\ref{lm:detect}.

\begin{proof}[Proof of Lemma~\ref{lm:detect}]

By assumption~\eqref{eq:r} we have
\begin{equation*}
\exp\lk-\frac{\pi^2}{r^2}\rk
\ge
\frac{2\log u}{n}.
\end{equation*}
Combining this with Fact~\ref{prop:concentration} gives
\begin{equation*}
\Prob\{F_r\lk \X\rk\le 0\}
\le \frac{1}{u}.
\end{equation*}
It follows that
\begin{equation*}
\Prob\left\{\X\in \mA_r\right\}=
\Prob\lk F_r\lk \X\rk>0\rk=
1-\Prob\lk F_r\lk \X\rk\le 0\rk\ge1-\frac{1}{u}.
\end{equation*}
On the other hand, by Fact~\ref{prop:flip}, we obtain
\begin{equation*}
\Prob\left\{\X\in \mA_r-r\mT_n\right\}
=\Prob\lk F_r\lk \X\rk<0\rk\le
\Prob\lk F_r\lk \X\rk\le 0\rk\le \frac{1}{u}.
\end{equation*}
Thus~\eqref{eq:measure} holds.
\end{proof}

\subsection{Proof of Theorem~\ref{thm:main}}

Applying Lemma~\ref{lm:detect} with $u=10$ shows that, for all sufficiently large $n$, every $r\ge \frac{\pi}{\sqrt{\log\lk \frac{n}{2\log 10}\rk}}$ is detectable. 
Therefore, by the definition of the detectability radius in~\eqref{eq:radius},
\begin{equation*}
r\lk\mT_n\rk\le
\frac{\pi}{\sqrt{\log\lk \frac{n}{2\log 10}\rk}}\lesssim \frac{1}{\sqrt{\log n}}.
\end{equation*}
Combining this estimate with Lemma~\ref{lm:focused}, which gives $\we\lk\mT_n\rk\asymp 1$, yields
\begin{equation*}
\frac{\we\lk\mT_n\rk }{r\lk\mT_n\rk}
\gtrsim
\sqrt{\log n}.
\end{equation*}
This proves Theorem~\ref{thm:main}.

\section{Non-Gaussian Data: A Polynomial Separation}\label{sec:nongaussian}

In Section~\ref{sec:main}, we showed that for Gaussian data the focused width can exceed the detectability radius by a factor of order at least $\sqrt{\log n}$.
A natural question is whether this $\sqrt{\log n}$ separation is intrinsic, in the sense that the separation factor cannot be improved beyond logarithmic order. 
In this section, we show that this is not the case once one moves beyond the Gaussian setting. 
We consider product data distributions, so that the coordinates are independent. 
In particular, for product Laplace data we obtain a stronger separation at least a factor of order $n^{1/4}$.

Let $\X$ be a random vector in $\R^n$. 
For a set $\mT\subset\R^n$, we say that a scale $r>0$ is $\X$-detectable if there exists a measurable set $\mA\subset\R^n$ such that
\begin{equation}\label{eq:X_detectable}
        \Prob\left\{\X\in\mA\right\}\ge 0.9,
        \qquad
        \Prob\left\{\X\in\mA-r\mT\right\}\le 0.1.
\end{equation}
The corresponding detectability radius is defined by
\begin{equation}\label{eq:X_radius}
        r_{\X}\lk\mT\rk
        :=
        \sup\left\{ r>0: \text{there is no measurable }\mA\subset\R^n \text{ satisfying}~\eqref{eq:X_detectable}\right\}.
\end{equation}
When $\X\sim N\lk0,\pmb{I}_n\rk$, this reduces to the Gaussian detectability radius $r\lk\mT\rk$ in~\eqref{eq:radius}.

Following the distribution-dependent formulation of focused width in~\cite{MPV}, we replace the Gaussian width by the width induced by the genuine data distribution. 
For a random vector $\X$ in $\R^n$ and a set $\mS\subset\R^n$, define
\begin{equation*}
        w_{\X}\lk\mS\rk:=\E\sup_{\vs\in\mS}\lg \X,\vs\rg.
\end{equation*}
The focused $\X$-width of $\mT$ is then defined by
\begin{equation}
        \we_{\X}\lk\mT\rk:=\inf_{\mS} w_{\X}\lk\mS\rk,
\end{equation}
where the infimum is taken over all origin-symmetric sets $\mS\subset\R^n$ satisfying the hitting condition
\begin{equation}\label{eq:hit2}
        \forall\,\vt\in\mT\quad \exists\,\vs\in\mS
        \quad\text{such that}\quad
        \lg \vt,\vs\rg\ge 1.
\end{equation}
When $\X\sim N\lk0,\pmb{I}_n\rk$, this reduces to the focused width $\we\lk\mT\rk$ defined in~\eqref{eq:focused_width}.

We next introduce the distribution-dependent scale at which the periodic cosine test becomes effective.
Let $\mu$ be a probability measure on $\R$ and $X\sim\mu$.
Define
\begin{equation}\label{eq:threshold}
        \rho_{n}
        :=
        \inf\left\{ \rho>0: \text{ for every } r\ge \rho,\lv \E\cos\lk \frac{\pi X}{r}\rk\rv\ge\sqrt{\frac{2\log 10}{n}} \right\}.
\end{equation}
This quantity is the threshold such that, for every $r$ above it, the periodic test succeeds with error probability at most $0.1$.

The following theorem gives the corresponding estimate for product data distributions:
the detectability radius is bounded by the threshold $\rho_n$, while the focused $\X$-width equals the first absolute moment.

\begin{theorem}\label{thm:product}
Let $\mu$ be a symmetric probability measure on $\R$ and $X\sim\mu$.
Let $a:=\E\lv X\rv<\infty$.
For each $n$, let $\X\sim\mu^{\otimes n}$, and let $\mT_n\subset\R^n$ be an origin-symmetric set satisfying~\eqref{eq:set}. 
Then, for all sufficiently large $n$,
\begin{equation*}
        r_{\X}\lk\mT_n\rk\le\rho_{n},
        \qquad
        \we_{\X}\lk\mT_n\rk=a.
\end{equation*}
Consequently,
\begin{equation*}
        \frac{\widetilde w_{\X}\lk\mT_n\rk}
        {r_{\X}\lk\mT_n\rk}
        \ge
        \frac{a}{\rho_{n}}.
\end{equation*}
\end{theorem}

We now specialize Theorem~\ref{thm:product} to product Laplace data. 
In this case, the distribution-dependent threshold is at most of order $n^{-1/4}$, 
and hence the ratio between the focused $\X$-width and the detectability radius is at least of order $n^{1/4}$.

\begin{corollary}\label{cor:laplace}
Let $\mu$ be the centered Laplace distribution with variance 1, namely
\begin{equation*}
        d\mu\lk x\rk= \frac{1}{\sqrt2}\exp\lk-\sqrt2\lv x\rv\rk\,dx.
\end{equation*}
Let $\X\sim\mu^{\otimes n}$, and let $\mT_n\subset\R^n$ be an origin-symmetric set satisfying~\eqref{eq:set}. 
Then, for all sufficiently large $n$,
\begin{equation*}
        \frac{\widetilde w_{\X}\lk\mT_n\rk}{r_{\X}\lk\mT_n\rk}
        \gtrsim
        n^{1/4}.
\end{equation*}
\end{corollary}

\subsection{Proof of Theorem~\ref{thm:product}}


We first prove the focused $\X$-width identity. 
The argument is the same as in Lemma~\ref{lm:focused}.
Since $\mu$ is symmetric and $\X\sim\mu^{\otimes n}$, we may write
\begin{equation*}
        \X\stackrel{d}{=}\veta\odot\pmb{\alpha},
\end{equation*}
where $\veta=\lk\eta_1,\ldots,\eta_n\rk$ is a random Rademacher vector independent of $\pmb{\alpha}$, and $\pmb{\alpha}$ has the same distribution as the coordinatewise absolute value of $\X$.  
In particular, for $j=1,\ldots,n$, one has $\E \alpha_j=a$.
For the upper bound, take $\mS_0=n^{-1}\mQ_n$.
As in the proof of Lemma~\ref{lm:focused}, this set satisfies the hitting condition~\eqref{eq:hit2} for every $\mT_n\subset\lk2\mathbb{Z}+1\rk^n$.
Therefore
\begin{equation}
        \we_{\X}\lk\mT_n\rk
        \le
        w_{\X}\lk\mS_0\rk=\frac{1}{n}\E\sup_{\veps\in\mQ_n}\lg \X,\veps\rg
        =\frac{1}{n}\E\sum_{j=1}^n \lv X_j\rv=a.
\end{equation}
For the lower bound, let $\mS$ be any origin-symmetric set satisfying the hitting condition for $\mT_n$, and define the support function
\begin{equation*}
        h_{\mS}\lk\x\rk:=\sup_{\vs\in\mS}\lg \x,\vs\rg.
\end{equation*}
Since $\mQ_n\subset\mT_n$, we have $h_{\mS}\lk\veps\rk\ge1$ for every $\veps\in\mQ_n$.
As in Lemma~\ref{lm:focused}, conditioning on $\veta$ and applying Jensen's inequality gives
\begin{equation*}
        \E_{\pmb{\alpha}}h_{\mS}\lk\veta\odot\pmb{\alpha}\rk
        \ge
        h_{\mS}\lk a\veta\rk=a h_{\mS}\lk\veta\rk\ge a.
\end{equation*}
Thus $w_{\X}\lk\mS\rk\ge a$.
Taking the infimum over all origin-symmetric sets $\mS$ satisfying the hitting condition~\eqref{eq:hit2} yields
\begin{equation}
        \we_{\X}\lk\mT_n\rk\ge a.
\end{equation}

It remains to prove the detectability estimate. 
Fix $r>\rho_{n}$, and set $\nu_r:=\E\cos\lk \frac{\pi X}{r}\rk$.
Define the signed testing function by
\begin{equation*}
        F_r\lk\x\rk
        :=
        \operatorname{sgn}\lk \nu_r\rk\sum_{j=1}^n\cos\lk\frac{\pi x_j}{r}\rk,
\end{equation*}
and let the acceptance region be
\begin{equation*}
        \mA_r:=\left\{\x\in\R^n: F_r\lk\x\rk>0\right\}.
\end{equation*}
As in Fact~\ref{prop:flip}, for every $\vt\in\lk2\mathbb{Z}+1\rk^n$, we obtain $F_r\lk\x+r\vt\rk=-F_r\lk\x\rk$,
and hence,
\begin{equation*}
        \mA_r-r\mT_n
        =
        \left\{\x\in\R^n:  F_r\lk\x\rk<0\right\}.
\end{equation*}
It remains only to estimate the probability of the event $\left\{  F_r\lk\X\rk\le 0\right\}$. 
The random variables $\left\{Z_j\right\}_{j=1}^n$, defined by $Z_j := \operatorname{sgn}\lk \nu_r\rk \cos\lk \frac{\pi X_j}{r}\rk$, are independent, take values in $[-1,1]$, and satisfy $\E Z_j=\lv \nu_r\rv$.
Therefore Hoeffding's inequality~\cite[Theorem 2.2.6]{Vershynin2018} gives
\begin{equation*}
\begin{aligned}
        \Prob\lk F_r\lk\X\rk\le0\rk
        &=\Prob\lk\sum_{j=1}^n\lk Z_j-\lv\nu_r\rv\rk\le -n\lv\nu_r\rv\rk  \\
        &\le \exp\lk-\frac{n\lv\nu_r\rv^2}{2}\rk
        \le \frac{1}{10},
\end{aligned}
\end{equation*}
where the last inequality follows from the definition of $\rho_n$ in~\eqref{eq:threshold}.
Consequently,
\begin{equation*}
        \Prob\left\{\X\in\mA_r\right\}\ge0.9,
        \qquad
        \Prob\left\{\X\in\mA_r-r\mT_n\right\}\le 0.1.
\end{equation*}
Thus every $r>\rho_{n}$ is $\X$-detectable for $\mT_n$, and hence
\begin{equation}\label{eq:rho}
        r_{\X}\lk\mT_n\rk\le \rho_{n}.
\end{equation}  

Combining~\eqref{eq:rho} with the focused $\X$-width identity $\we_{\X}\lk\mT_n\rk=a$ gives
\begin{equation*}
        \frac{\we_{\X}\lk\mT_n\rk}{r_{\X}\lk\mT_n\rk}
        \ge
        \frac{a}{\rho_{n}}.
\end{equation*}

\subsection{Proof of Corollary~\ref{cor:laplace}}

The characteristic function of the centered, variance 1 Laplace distribution $X$ is
\begin{equation*}
        \widehat\mu\lk u\rk=\frac{1}{1+u^2/2}.
\end{equation*}
Hence, for every $r>0$,
\begin{equation}\label{eq:mu}
        \E\cos\lk \frac{\pi X}{r}\rk=\widehat\mu\lk \frac{\pi}{r}\rk =\frac{1}{1+\pi^2/(2r^2)}.
\end{equation}
Since the quantity in~\eqref{eq:mu} is positive and increasing in $r$, the definition of $\rho_n$ in~\eqref{eq:threshold} gives, for all sufficiently large $n$,
\begin{equation}\label{eq:rho0}
        \rho_n=\frac{\pi}{\sqrt2}\lk\sqrt{\frac{n}{2\log 10}}-1\rk^{-1/2}\lesssim n^{-1/4}.
\end{equation}

Applying Theorem~\ref{thm:product} and the estimate~\eqref{eq:rho0}, we obtain
\begin{equation*}
        r_{\X}\lk\mT_n\rk\le\rho_n \lesssim n^{-1/4}.
\end{equation*}
Moreover, for the centered, variance 1 Laplace distribution, $\E\lv X\rv =\frac{1}{\sqrt2}$.
Thus, Theorem~\ref{thm:product} also gives $\we_{\X}\lk\mT_n\rk =\frac{1}{\sqrt2}$.
Consequently,
\begin{equation*}
        \frac{\widetilde w_{\X}\lk\mT_n\rk}{r_{\X}\lk\mT_n\rk}
        \gtrsim n^{1/4}.
\end{equation*}


\section{Summary}

This note establishes two separations between focused width and detectability radius.
In the Gaussian model, for every origin-symmetric perturbation set $\mT_n$ satisfying
\begin{equation*}
\mQ_n\subset \mT_n\subset \lk2\mathbb{Z}+1\rk^n,
\end{equation*}
the focused width can overestimate the detectability radius by at least a factor of order $\sqrt{\log n}$.
For the same class of perturbation sets, the product Laplace model exhibits a stronger phenomenon: the focused $\X$-width can exceed the detectability radius by at least a polynomial factor of order $n^{1/4}$.
Thus the logarithmic scale in the Gaussian result is not intrinsic to the adversarial fake detection problem itself.

The reason focused width fails in these examples is that it is a linear geometric quantity: it measures the complexity of linear witnesses satisfying the hitting condition, whereas the detectability radius is defined through arbitrary measurable tests, which may exploit nonlinear, nonconvex, or nonmonotone structure.
For the perturbation sets considered here, the relevant structure is arithmetic: all admissible shifts lie in the odd integer grid, allowing periodic tests to distinguish shifted samples at scales much smaller than the focused width.
This does not contradict the positive results of~\cite{MPV} in the highly symmetric regime, since the sets considered here, even $\mQ_n$, contain only fixed-magnitude sign vectors and no radial enlargements required by high symmetry.

The results of this note also suggest a broader difficulty: there may be no single universal mechanism for adversarial fake detection.
For highly symmetric sets, the methods of~\cite{MPV} identify the correct geometric scale; for the hypercube and odd integer grid considered here, detection is instead driven by a periodic test adapted to the arithmetic structure of the perturbations.
For general perturbation sets, however, it remains unclear how to construct tests that simultaneously capture their geometry, arithmetic structure, data distribution, and sample-dependent adversarial choices.
This absence of a universal testing mechanism is one of the main challenges in adversarial fake data detection, especially in modern settings where fake data can be generated adaptively.

\normalem
\bibliographystyle{plain}

\bibliography{ref}

\end{document}